\documentclass[reqno]{amsart}
\usepackage{amssymb}
\usepackage{amsmath}
\usepackage{color}

\newtheorem{theorem}{Theorem}[section]
\newtheorem{lemma}[theorem]{Lemma}

\newtheorem{remark}[theorem]{Remark}
\theoremstyle{definition}
\newtheorem{definition}[theorem]{Definition}

\numberwithin{equation}{section}

\newcommand{ \mint }{ {\int\hspace{-0.38cm}- }}

\newcommand{ \mr }{ \mathbb{R} }

\newcommand{\Norm}[1]{\left|\hspace{-0.3mm}\left| #1 \right|\hspace{-0.3mm}\right|}

\begin{document}

\title[Nondivergence problems with irregular obstacles]
{Nondivergence elliptic and parabolic problems with irregular obstacles}

\author{Sun-Sig Byun}
\address{Department of Mathematical Sciences and Research Institute of Mathematics,
Seoul National University, Seoul 08826, Korea}
\email{byun@snu.ac.kr}

\author{Ki-Ahm Lee}
\address{Department of Mathematical Sciences,
Seoul National University, Seoul 08826, Korea.
Center for Mathematical Challenges,
Korea Institute for Advanced Study, Seoul 02455, Korea}
\email{kiahm@snu.ac.kr}

\author{Jehan Oh}
\address{Department of Mathematical Sciences,
Seoul National University, Seoul 08826, Korea}
\email{ojhan0306@snu.ac.kr}

\author{Jinwan Park}
\address{Department of Mathematical Sciences,
Seoul National University, Seoul 08826, Korea}
\email{jinwann@snu.ac.kr}

\thanks{This work was supported by the National Research Foundation of Korea (NRF) grant funded by the Korea Government (NRF-2015R1A4A1041675).}


\date{\today.}


\begin{abstract}
We prove the natural weighted Calder\'{o}n and Zygmund estimates for solutions to elliptic and parabolic obstacle problems in nondivergence form with discontinuous coefficients and irregular obstacles.
We also obtain Morrey regularity results for the Hessian of the solutions and H\"{o}lder continuity of the gradient of the solutions.
\end{abstract}

\maketitle

\tableofcontents

\section{\bf Introduction}
\label{sec1}

We study in this paper the following elliptic obstacle problems:
\begin{equation}
\label{obs prob lin}
\left\{\begin{array}{rclcc}\
a_{ij}(x) D_{ij}u & \leq & f & \mathrm{in} &  \Omega,\\
\left( a_{ij}(x) D_{ij}u - f \right) (u - \psi) & = & 0 & \mathrm{in} &  \Omega,\\
u & \geq & \psi & \mathrm{in} &  \Omega,\\
u & = & 0 & \mathrm{on} & \partial \Omega,
\end{array}\right.
\end{equation}
and
\begin{equation}
\label{obs prob fully}
\left\{\begin{array}{rclcc}\
F(x, D^2 u) & \leq & f & \mathrm{in} &  \Omega,\\
\left( F(x, D^2 u) - f \right) (u - \psi) & = & 0 & \mathrm{in} &  \Omega,\\
u & \geq & \psi & \mathrm{in} &  \Omega,\\
u & = & 0 & \mathrm{on} & \partial \Omega.
\end{array}\right.
\end{equation}
Here $\Omega$ is a bounded domain in $\mr^n$, $n \geq 2$, with its boundary $\partial \Omega\in C^{1,1}$.
The coefficient matrix $(a_{ij}(x))$ and the fully nonlinear operator $F(x,M)$ are supposed to be uniformly elliptic, see Section \ref{sec2}.
The nonhomogeneous term $f \in L^p_w(\Omega)$ is given, as is the obstacle function $\psi \in W^{2,p}_w(\Omega)$, $\psi \leq 0$ a.e. on $\partial \Omega$, where a weight $w$ in some Muckenhoupt class and the range of $p$ will be clarified later.

We also consider the following parabolic obstacle problem:
\begin{equation}
\label{parabolic obs prob}
\left\{\begin{array}{rclcc}\
u_t - a_{ij}(x,t) D_{ij}u & \geq & f & \mathrm{in} & \Omega_T,\\
\left( u_t - a_{ij}(x,t) D_{ij}u - f \right) (u - \psi) & = & 0 & \mathrm{in} & \Omega_T,\\
u & \geq & \psi & \mathrm{in} & \Omega_T,\\
u & = & 0 & \mathrm{on} & \partial_p \Omega_T,
\end{array}\right.
\end{equation}
where $\Omega_T:= \Omega \times (0,T]$, $T>0$, and $\partial_p \Omega_T:= (\partial \Omega \times [0,T]) \cup (\Omega \times \{t=0\})$ with $\partial \Omega\in C^{1,1}$.
Here the coefficient matrix $(a_{ij}(x,t))$ is uniformly parabolic, see Section \ref{sec2}, the nonhomogeneous term $f$ is in $L^p_w(\Omega_T)$ with $p>2$ and $w$ in a Muckenhoupt class, and the obstacle function is $\psi\in W^{2,1}L^p_w(\Omega_T)$ with $\psi \leq 0$ a.e. on $\partial_p \Omega_T$.

The main purpose of this study is to investigate existence, uniqueness and regularity properties of solutions to the obstacle problems (\ref{obs prob lin}), (\ref{obs prob fully}) and (\ref{parabolic obs prob}) in the framework of weighted Lebesgue spaces.
The weighted Lebesgue spaces $L^p_w$ not only generalize the classical Lebesgue spaces $L^p$, but also are closely related to Morrey spaces $L^{p,\theta}$.
In particular, knowing the fact that the Hardy-Littlewood maximal function of the characteristic function of a ball is a Muckenhoupt weight (see \cite{CR}), we are able to obtain an optimal Morrey regularity for the Hessian of the solutions to (\ref{obs prob lin}) and (\ref{obs prob fully}).
This leads to a higher integrability result of the Hessian of the solutions and H\"{o}lder continuity of the gradient of the solutions.

In this paper we deal with discontinuous coefficients $a_{ij}$, irregular obstacle functions $\psi$ and discontinuous nonhomogeneous terms $f$ given in the weighted Lebesgue spaces.
We notice that if $\partial \Omega$, $a_{ij}$, $f$ and $\psi$ are smooth enough, for instance, $\partial \Omega \in C^{2,\alpha}$, $a_{ij},f \in C^{\alpha}(\overline{\Omega})$ for some $\alpha>0$, and $\psi \in C^2 (\overline{\Omega})$, then the obstacle problem (\ref{obs prob lin}) has a unique strong solution $u \in C^{1,1}(\overline{\Omega})$, see \cite{Fri}, and furthermore, the $C^{1,1}$ regularity of solutions for various types of obstacle problems has been extensively investigated under appropriate regularity assumptions on the boundary of domain, the obstacle, the nonhomogeneous term, see \cite{FS, IM, PSU}.

In the case of discontinuous coefficients and irregular nonhomogeneous terms, but without obstacles, the regularity results for elliptic and parabolic equations in nondivergence form have been obtained in \cite{Ca, CC, CFL1, CFL2, Es, Wi} for the elliptic case, and in \cite{BC1, BP1, HHH, WY1} for the parabolic case.
In particular, weighted $W^{2,p}$ estimates were established in a series of papers \cite{BL1, BL2, BLP}.
Here we want to extend these results for nondivergence structure problems from the non-obstacle case to the obstacle case.
More precisely, we shall establish the weighted $W^{2,p}$ estimates of solutions to the elliptic obstacle problems (\ref{obs prob lin}) and (\ref{obs prob fully}), and parabolic obstacle problem (\ref{parabolic obs prob}), by essentially proving that the Hessian of solutions is as regular as the nonhomogeneous terms and the Hessian of the associated obstacles.

Our approach is mainly based on a new general approximation argument in the literature.
Unlike other approximation arguments which have in general penalty terms as in \cite{Fri}, we find a better approximation of Heaviside functions in order to use specially redesigned reference equations (\ref{fully approx}), (\ref{lin approx}) and (\ref{lin approx-parabolic}).
The choice of such an approximation method seems to be appropriate to our theory, as the problem under consideration is in the setting of Lebesgue spaces and one can easily control the $L^p$-norm of the nonhomogeneous term in a reference equation. Although this approximation method does not involve penalty terms, we can utilize comparison principles to show that the solution is in the constraint set.
Furthermore, this approach can be extended to the fully nonlinear obstacle problems.

This paper is organized as follows.
In the next section we introduce some background and review weighted Lebesgue and Sobolev spaces.
In Section \ref{sec3} and \ref{sec4} we establish the weighted $W^{2,p}$ estimates for the elliptic fully nonlinear obstacle problem (\ref{obs prob fully}) and elliptic linear obstacle problem (\ref{obs prob lin}), respectively.
In section \ref{sec5} we present Morrey regularity results and obtain H\"{o}lder continuity of the gradient of the solutions for the elliptic obstacle problems.
Finally, in the last section we prove the weighted $W^{2,p}$ estimates for parabolic linear obstacle problem (\ref{parabolic obs prob}).

\section{\bf Preliminaries}
\label{sec2}

\subsection{Notations}

We start with some standard notations and terminologies.
\begin{enumerate}
\item For $y \in \mr^n$ and $r>0$, $B_r(y):=\{ x \in \mr^n:|x-y|<r \}$ denotes the open ball in $\mr^n$ with center $y$ and radius $r$.
\item For $(y,s) \in \mr^n \times \mr$ and $r>0$, $Q_r(y,s):=B_r(y) \times (s-r^2,s+r^2]$ denotes the parabolic cylinder with middle center $(y,s)$, radius $r$, height $r^2$.
\item For a Lebesgue measurable set $E \subset \mr^n$, $|E|$ denotes the Lebesgue measure of $E$.
\item  For an integrable function $h:E \rightarrow \mr$ with a bounded measurable set $E \subset \mr^n$, we denote $\overline{h}_E$ the integral average of $h$ on $E$ by
\begin{equation*}
\overline{h}_E := \mint_E h(x) \, dx = \frac{1}{|E|} \int_E h(x) \ dx.
\end{equation*}
\item $\langle \cdot , \cdot \rangle : \mr^n \times \mr^n \rightarrow \mr$ denotes the Euclidean inner product in $\mr^n$.
\item $S(n)$ denotes the set of real $n\times n$ symmetric matrices.
For $M \in S(n)$, $\Norm{M}$ denotes the $(L^2,L^2)$-norm of $M$, that is, $\displaystyle \Norm{M} = \sup_{|x|=1}|Mx|$, and we write $M \geq 0$ to mean that $M$ is a non-negative definite symmetric matrix.
\item The summation convention of repeated indices are used.
\item For the sake of convenience, we employ the letter $c$ to denote any universal constants which can be explicitly computed in terms of known quantities, and so $c$ might vary from line to line.
\end{enumerate}

\subsection{Basic assumptions}

For the problem (\ref{obs prob lin}), the coefficient matrix $\mathbf{A}=(a_{ij}) : \mr^n \rightarrow \mr^{n \times n}$ is assumed to be symmetric (that is, $a_{ij} \equiv a_{ji}$) and uniformly elliptic in the following sense:

\begin{definition}
We say that the coefficient matrix $\mathbf{A}$ is \textit{uniformly elliptic} if there exist positive constants $\lambda$ and $\Lambda$ such that
\begin{equation}
\lambda |\xi|^2 \leq \langle \mathbf{A}(x)\xi, \xi \rangle \leq \Lambda |\xi|^2
\end{equation}
for almost every $x \in \mr^n$ and all $\xi \in \mr^n$.
\end{definition}

For the problem (\ref{obs prob fully}), the fully nonlinear operator $F=F(x,M)$ is assumed to be uniformly elliptic in the following sense:

\begin{definition}
We say that the fully nonlinear operator $F$ is \textit{uniformly elliptic} if there exist positive constants $\lambda$ and $\Lambda$ such that
\begin{equation}
\label{uniformly elliptic}
\lambda \Norm{N} \leq F(x,M+N) - F(x,M) \leq \Lambda \Norm{N}
\end{equation}
for almost every $x \in \Omega$ and all $M, N \in S(n)$ with $N \geq 0$.
\end{definition}

We also assume that $F(x,0) \equiv 0$, for simplicity, and that $F=F(x,M)$ is a convex function of $M \in S(n)$.

\vspace{0.2cm}

For the parabolic problem (\ref{parabolic obs prob}), the coefficient matrix $\mathbf{A}=(a_{ij}) : \mr^n \times \mr \rightarrow \mr^{n \times n}$ is assumed to be symmetric and uniformly parabolic in the following sense:

\begin{definition}
We say that the coefficient matrix $\mathbf{A}=\mathbf{A}(x,t)$ is \textit{uniformly parabolic} if there exist positive constants $\lambda$ and $\Lambda$ such that
\begin{equation}
\lambda |\xi|^2 \leq \langle \mathbf{A}(x,t)\xi, \xi \rangle \leq \Lambda |\xi|^2
\end{equation}
for almost every $(x,t) \in \mr^n \times \mr$ and all $\xi \in \mr^n$.
\end{definition}

\subsection{Weighted Lebesgue and Sobolev spaces}

\begin{definition}
Let $1<s<\infty$. We say that $w$ is a \textit{weight} in \textit{Muckenhoupt class $A_s$}, or an \textit{$A_s$ weight}, if $w$ is a locally integrable nonnegative function on $\mr^n$ with
\begin{equation}
[w]_s := \sup_{B} \left( \mint_{B} w(x) \, dx \right) \left( \mint_{B} w(x)^{-\frac{1}{s-1}} \, dx \right)^{s-1} < +\infty,
\end{equation}
where the supremum is taken over all balls $B \subset \mr^n$.
If $w$ is an $A_s$ weight, we write $w \in A_s$, and $[w]_s$ is called the \textit{$A_s$ constant of $w$}.
\end{definition}

We note that the Muckenhoupt classes $A_s$ are monotone in $s$, more precisely, $A_{s_1} \subset A_{s_2}$ for $1 < s_1 \leq s_2 <
\infty$.

The \emph{weighted Lebesgue space} $L^p_w(\Omega)$, $1<p<\infty$, $w \in A_s$ with $1<s<\infty$, consists of all measurable functions $g$ on $\Omega$ such that
\begin{equation*}
\Norm{g}_{L^p_w(\Omega)}:=\left( \int_\Omega |g|^p w \ dx \right)^{\frac{1}{p }}< +\infty.
\end{equation*}

The \emph{weighted Sobolev space} $W^{m,p}_w(\Omega)$, $m \in \mathbb{N}$, $1<p<\infty$, $w\in A_s$ with $1<s<\infty$, is defined by a class of functions $g\in L^p_w(\Omega)$ with weak derivatives $D^\alpha g\in L^p_w(\Omega)$ for all multiindex $\alpha$ with $|\alpha| \leq m$.
The norm of $g$ in $W^{m,p}_w(\Omega)$ is defined by
\begin{equation*}
\Norm{g}_{  W^{m,p}_w(\Omega)}:=\left( \sum_{|\alpha| \leq m}\int_\Omega |D^\alpha g|^p w \ dx \right)
^{\frac{1}{p}}.
\end{equation*}

For a detailed discussion of the weighted Lebesgue and Sobolev spaces, we refer the readers to \cite{St, Tu} and references therein.
We will use the following embedding lemma later in the proof of Theorem \ref{fully-obs-thm}, see \cite[Remark 2.4]{BLP}.

\begin{lemma}
\label{embedding-lemma}
Let $n_0 < p < \infty$ for some $n_0 > 1$, and let $w \in A_{\frac{p}{n_0}}$.
Suppose that $f \in L^p_w(\Omega)$.
Then $f \in L^{\frac{p n_0}{p - n_0 \kappa}}(\Omega)$ for some small $\kappa = \kappa \left(n,\frac{p}{n_0},[w]_{\frac{p}{n_0}} \right) > 0$ with the estimate
\begin{equation}
\Norm{f}_{L^{\frac{p n_0}{p - n_0 \kappa}}(\Omega)} \leq c \Norm{f}_{L^p_w(\Omega)},
\end{equation}
for some positive constant $c=c(n,n_0,p,[w]_{\frac{p}{n_0}},\mathrm{diam}(\Omega))$.
\end{lemma}

\section{\bf Elliptic fully nonlinear obstacle problems}
\label{sec3}

In order to measure the oscillation of $F=F(M,x)$ with respect to
the variable $x$, we define
\begin{equation*}
\beta_F(x,x_0) := \sup_{M \in S(n) \setminus \{ 0 \}} \frac{|F(x,M)-F(x_0,M)|}{\Norm{M}},
\end{equation*}
and set $\beta(x,x_0):=\beta_F(x,x_0)$ for the sake of simplicity.

We first need the following weighted $W^{2,p}$ estimate for convex fully nonlinear equations without obstacle. This can be found in \cite{BLP}.

\begin{lemma}
\label{fully nonlinear eq}
Let $n_0 < p < \infty$, where $n_0:=n-\nu_0$ for some $\nu_0 = \nu_0 \left( \frac{\Lambda}{\lambda},n \right) > 0$, and let $w \in A_{\frac{p}{n_0}}$.
Suppose that $\partial \Omega \in C^{1,1}$ and $f \in L^p_w(\Omega)$. Then there exists a small $\delta=\delta(n,\lambda,\Lambda,p,w,\partial \Omega)>0$ such that if
\begin{equation}
\label{con beta}
\sup_{x_0 \in \overline{\Omega}, 0 < r \leq R_0} \left( \mint_{B_r(x_0) \cap \Omega} \beta(x,x_0)^n \, dx \right)^{\frac{1}{n}} \leq \delta
\end{equation}
for some $R_0 > 0$, then the problem
\begin{equation}
\label{fully prob}
\left\{\begin{array}{rclcc}\
F(x, D^2 u) & = & f & \mathrm{in} & \Omega,\\
u & = & 0 & \mathrm{on} & \partial \Omega,
\end{array}\right.
\end{equation}
has a unique solution $u \in W^{2,p}_w(\Omega)$ with the estimate
\begin{equation}
\Norm{u}_{W^{2,p}_w(\Omega)} \leq c \Norm{f}_{L^p_w(\Omega)},
\end{equation}
for some positive constant $c=c(n,\lambda,\Lambda,p,w,\partial \Omega,\mathrm{diam}(\Omega),R_0)$.
\end{lemma}

We will use the following comparison principle for fully nonlinear operators, see \cite[Theorem 2.10]{CCKS}.

\begin{lemma}
\label{comparison-fully}
Suppose that $U$ is a bounded domain and that $f \in L^p(U)$, $1<p<\infty$.
Let $u_1, u_2 \in C(\overline{U})$ be supersolution and subsolution of the equation $F(x,D^2 u)=f$ in $U$, respectively, with $u_1 \geq u_2$ in $\partial U$.
Then we have $u_1 \geq u_2$ in $U$.
\end{lemma}

We now state the first main result in this paper, the weighted $W^{2,p}$ estimate for the obstacle problem (\ref{obs prob fully}).

\begin{theorem}[Main Theorem 1]
\label{fully-obs-thm}
Let $n_0 < p < \infty$, where $n_0:=n-\nu_0$ for some $\nu_0 = \nu_0 \left( \frac{\Lambda}{\lambda},n \right) > 0$, and let $w \in A_{\frac{p}{n_0}}$.
Suppose that $\partial \Omega \in C^{1,1}$, $f \in L^p_w(\Omega)$ and $\psi \in W^{2,p}_w(\Omega)$.
Then there exists a small $\delta=\delta(n,\lambda,\Lambda,p,w,\partial \Omega)>0$ such that if
\begin{equation}
\sup_{x_0 \in \overline{\Omega}, 0 < r \leq R_0} \left( \mint_{B_r(x_0) \cap \Omega} \beta(x,x_0)^n \, dx \right)^{\frac{1}{n}} \leq \delta
\end{equation}
for some $R_0 > 0$, then the fully nonlinear obstacle problem (\ref{obs prob fully}) has a unique solution $u \in W^{2,p}_w(\Omega)$ with the estimate
\begin{equation}
\Norm{u}_{W^{2,p}_w(\Omega)} \leq c \left( \Norm{f}_{L^p_w(\Omega)} + \Norm{\psi}_{W^{2,p}_w(\Omega)} \right),
\end{equation}
for some positive constant $c=c(n,\lambda,\Lambda,p,w,\partial \Omega,\mathrm{diam}(\Omega),R_0)$.
\end{theorem}

\begin{proof}
First, in order to approximate the Heaviside function, we choose a non-decreasing smooth function $\Phi_{\varepsilon} \in C^{\infty}(\mr)$, see for instance \cite{Te, Ur}, satisfying
\begin{equation*}
\Phi_{\varepsilon}(s) \equiv 0 \quad \mathrm{if} \quad s \leq 0; \qquad \Phi_{\varepsilon}(s) \equiv 1 \quad \mathrm{if} \quad s \geq \varepsilon,
\end{equation*}
and
\begin{equation*}
0 \leq \Phi_{\varepsilon}(s) \leq 1, \ \forall s \in \mr.
\end{equation*}
Let $g(x):=f(x)-F(x, D^2 \psi(x))$ for $x \in \Omega$.
Since $f \in L^p_w(\Omega)$, $\psi \in W^{2,p}_w(\Omega)$ and $F$ is Lipschitz in $M$, we have $g \in L^p_w(\Omega)$ with the estimate
\begin{align*}
\Norm{g}_{L^p_w(\Omega)} & \leq c \left( \Norm{f}_{L^p_w(\Omega)} + \Norm{F(\cdot, D^2 \psi)}_{L^p_w(\Omega)} \right) \\
& \leq c \left( \Norm{f}_{L^p_w(\Omega)} + \Norm{\psi}_{W^{2,p}_w(\Omega)} \right).
\end{align*}
We write $g^+ = \max \{ g, 0 \}$ and $g^- = \max \{ -g, 0 \}$, and consider the following problem:
\begin{equation}
\label{fully approx}
\left\{\begin{array}{rclcc}\
F(x, D^2 u_{\varepsilon}) & = & g^+ \Phi_{\varepsilon}(u_{\varepsilon}-\psi) + f - g^+ & \mathrm{in} & \Omega,\\
u_{\varepsilon} & = & 0 & \mathrm{on} & \partial \Omega.
\end{array}\right.
\end{equation}
We note that the above problem (\ref{fully approx}) has a unique solution.
Indeed, according to Lemma \ref{fully nonlinear eq}, it follows that for each $v_0 \in L^p_w(\Omega)$, there is $v \in W^{2,p}_w(\Omega)$ satisfying
\begin{equation*}
\left\{\begin{array}{rclcc}\
F(x, D^2 v) & = & g^+ \Phi_{\varepsilon}(v_0-\psi) + f - g^+ & \mathrm{in} & \Omega,\\
v & = & 0 & \mathrm{on} & \partial \Omega.
\end{array}\right.
\end{equation*}
Since $0 \leq \Phi_{\varepsilon}(\cdot) \leq 1$, we can deduce from Lemma \ref{fully nonlinear eq} that
\begin{equation*}
\Norm{v}_{W^{2,p}_w(\Omega)} \leq R,
\end{equation*}
where $R$ is independent of $v_0$.
Defining $v=Sv_0$, we see that $S$ maps the $R$-ball in $L^p_w(\Omega)$ into itself and $S$ is compact, as $W^{2,p}_w(\Omega)$ is a compact subset of $L^p_w(\Omega)$ (see for instance \cite{Ki}).
By Schauder's fixed point theorem, there is $u_{\varepsilon}$ such that $u_{\varepsilon}=Su_{\varepsilon}$, which is the solution to the problem (\ref{fully approx}).

Now, it follows from Lemma \ref{fully nonlinear eq} that
\begin{align*}
\Norm{u_{\varepsilon}}_{W^{2,p}_w(\Omega)} & \leq c \left( \Norm{g^+  \Phi_{\varepsilon}(u_{\varepsilon}-\psi)}_{L^p_w(\Omega)} + \Norm{f}_{L^p_w(\Omega)} + \Norm{g^+}_{L^p_w(\Omega)} \right) \\
& \leq c \left( \Norm{g^+}_{L^p_w(\Omega)} + \Norm{f}_{L^p_w(\Omega)} + \Norm{g^+}_{L^p_w(\Omega)} \right) \\
& \leq c \left( \Norm{f}_{L^p_w(\Omega)} + \Norm{g}_{L^p_w(\Omega)} \right) \\
& \leq c \left( \Norm{f}_{L^p_w(\Omega)} + \Norm{\psi}_{W^{2,p}_w(\Omega)} \right).
\end{align*}
Hence $\{ u_{\varepsilon} \}$ is uniformly bounded in $W^{2,p}_w(\Omega)$.
Using Lemma \ref{embedding-lemma} and Sobolev embedding, we can find a subsequence $\{ u_{\varepsilon_k} \}_{k=1}^{\infty}$ with $\varepsilon_k \searrow 0$, and a function $u \in W^{2,p}_w(\Omega) \cap C^{\alpha}(\overline \Omega)$ such that $u_{\varepsilon_k}$ converges to $u$ weakly in $W^{2,p}_w(\Omega)$ and strongly in $C^{\alpha}(\overline{\Omega})$ for some $\alpha>0$.

Next, we claim that $u$ is a solution of the fully nonlinear obstacle problem (\ref{obs prob fully}).
First, we see that $u=0$ on the boundary $\partial \Omega$ since $u_{\varepsilon_k}$ uniformly converges to $u$ and $u_{\varepsilon_k}=0$ on $\partial \Omega$ for every $k$.
Also, we have from (\ref{fully approx}) that
\begin{equation*}
F(x, D^2 u_{\varepsilon_k}) = g^+ \Phi_{\varepsilon_k}(u_{\varepsilon_k}-\psi) - g^+ + f \leq f \ \quad \mathrm{in} \ \ \Omega,
\end{equation*}
for every $k$, and hence $F(x, D^2 u) \leq f$ in $\Omega$.

We now prove that $u \geq \psi$ in $\Omega$.
For each $m \in \mathbb{N}$, $\Phi_{\varepsilon_k}(u_{\varepsilon_k}-\psi)$ converges to $0$ on the set $\left\lbrace u < \psi - \frac{1}{m} \right\rbrace$, by the uniform convergence of the $u_{\varepsilon_k}$.
Therefore, $F(x, D^2 u) = f - g^+$ on the set $\left\lbrace u < \psi - \frac{1}{m} \right\rbrace$, for each $m \in \mathbb{N}$.
Since $\left\lbrace u < \psi \right\rbrace = \bigcup_{m=1}^{\infty} \left\lbrace u < \psi - \frac{1}{m} \right\rbrace$, we have $F(x, D^2 u) = f - g^+$ on the set $\left\lbrace u < \psi \right\rbrace$.
We note that $u, \psi \in C(\Omega)$, since $u, \psi \in W^{2,\widetilde{p}}(\Omega)$ for some $\widetilde{p} > \frac{n}{2}$ by Lemma \ref{embedding-lemma}.
Hence, $V:=\{u<\psi\}$ is an open set.
Now suppose that $V \neq \emptyset$.
From the definition of $g$, we have
\begin{equation*}
F(x, D^2 \psi) = f - g \ \quad \mathrm{in} \ \  V.
\end{equation*}
Also it is clear that
\begin{equation*}
F(x, D^2 u) = f - g^+\leq f - g \ \quad \mathrm{in} \ \ V,
\end{equation*}
and that
\begin{equation*}
u = \psi \ \quad \mathrm{on} \ \ \partial V.
\end{equation*}
Then we obtain from Lemma \ref{comparison-fully} that $u \geq \psi$ in $V$, which contradicts the definition of the set $V$.
We thus conclude that $V = \emptyset$ and $u \geq \psi$ in $\Omega$.

Finally, we prove that $F(x, D^2 u) = f$ on the set $\left\lbrace u > \psi \right\rbrace$.
To do this, observe that for each $m \in \mathbb{N}$, $\Phi_{\varepsilon_k}(u_{\varepsilon_k}-\psi)$ converges to $1$ almost everywhere on the set $\left\lbrace u > \psi + \frac{1}{m} \right\rbrace$.
Therefore, we obtain
\begin{equation*}
F(x, D^2 u) = g^+ + f - g^+ = f
\end{equation*}
on the set $\left\lbrace u > \psi \right\rbrace = \bigcup_{m=1}^{\infty} \left\lbrace u > \psi + \frac{1}{m} \right\rbrace$.

Consequently, $u \in W^{2,p}_w(\Omega)$ is a solution to (\ref{obs prob fully}) with the estimate
\begin{equation*}
\Norm{u}_{W^{2,p}_w(\Omega)} \leq \liminf_{k \to \infty} \Norm{u_{\varepsilon_k}}_{W^{2,p}_w(\Omega)} \leq c \left( \Norm{f}_{L^p_w(\Omega)} + \Norm{\psi}_{W^{2,p}_w(\Omega)} \right).
\end{equation*}

To show the uniqueness, let $u_1, u_2$ be two solutions to (\ref{obs prob fully}) and suppose that the open set $G := \{ u_2 > u_1 \}$ is nonempty.
Since $u_2 > u_1 \geq \psi$ in $G$, we know that $F(x, D^2 u_2) = f$ in $G$.
Therefore, we have
\begin{equation}
\left\{\begin{array}{rclcc}\
F(x, D^2 u_2) = f & \geq & F(x, D^2 u_1) & \mathrm{in} & G,\\
u_2 & = & u_1 & \mathrm{on} & \partial G.
\end{array}\right.
\end{equation}
By Lemma \ref{comparison-fully}, we get $u_2 \leq u_1$ in $G$, a contradiction.
Hence, the solution of (\ref{obs prob fully}) is unique.
This completes the proof.
\end{proof}

\section{\bf Elliptic linear obstacle problems}
\label{sec4}

We start with the small bounded mean oscillation (BMO) assumption on the coefficient matrix $\mathbf{A}$ for the linear obstacle problem (\ref{obs prob lin}).

\begin{definition}
We say that the coefficient matrix $\mathbf{A}$ is \textit{$(\delta,R)$-vanishing} if
\begin{equation}
\label{vanishing}
\sup_{0 < r \leq R} \sup_{y \in \mr^n} \left( \mint_{B_r(y)} |\mathbf{A}(x)-\overline{\mathbf{A}}_{B_r(y)}|^2 \, dx \right)^{\frac{1}{2}} \leq \delta,
\end{equation}
where $\overline{\mathbf{A}}_{B_r(y)} = \mint_{B_r(y)} \mathbf{A}(x) \, dx$ is the integral average of $\mathbf{A}$ on the ball $B_r(y)$.
\end{definition}

We note that one can take $R=1$ for simplicity, which is due to the scaling invariance property.
On the other hand, $\delta>0$ is invariant under such a scaling.
The assumption (\ref{vanishing}) on the coefficient matrix is weaker than the vanishing mean oscillation (VMO) or continuity assumption on the coefficient matrix, see \cite{BLP, BP1} for more details.

We next introduce the weighted $W^{2,p}$ estimates for linear elliptic equations without obstacle, see \cite{BL2}.

\begin{lemma}
\label{lin eq}
Let $2<p<\infty$ and let $w \in A_{\frac{p}{2}}$.
Suppose that $\partial \Omega \in C^{1,1}$ and $f \in
L^p_w(\Omega)$.
There exists a small $\delta=\delta(\Lambda,p,n,w,\partial \Omega)>0$ such that if $\mathbf{A}$ is uniformly elliptic and $(\delta,R)$-vanishing, then the problem
\begin{equation*}
\left\{ \begin{array}{rclcc}
a_{ij}D_{ij}u & = & f & \mathrm{in} & \Omega,\\
u & = & 0 & \mathrm{on} & \partial \Omega,
\end{array}\right.
\end{equation*}
has a unique solution $u \in W^{2,p}_w(\Omega)$ with the estimate
\begin{equation}
\Norm{u}_{W^{2,p}_w(\Omega)} \le c \Norm{f}_{L^p_w(\Omega)},
\end{equation}
for some positive constant $c=c(n,\lambda,\Lambda,p,w,\partial \Omega,\mathrm{diam}(\Omega))$.
\end{lemma}

We also need the following maximum principle for linear equations, see for instance \cite[Theorem 2.10]{CCKS} and \cite{CFL2}.

\begin{lemma}
\label{comparison-lin}
Suppose that $U$ is a bounded domain and that $\mathbf{A}=(a_{ij})$ is uniformly elliptic and $(\delta,R)$-vanishing.
If $u$ satisfies
\begin{equation*}
\left\{\begin{array}{rclcc}\
a_{ij}D_{ij}u & \leq & 0 & \mathrm{in} & U,\\
u & \geq & 0 & \mathrm{on} & \partial U,
\end{array}\right.
\end{equation*}
then $u \geq 0$ in $U$.
\end{lemma}

Now we state and prove the second main result in this paper, the global weighted $W^{2,p}$ estimate for the linear elliptic obstacle problem (\ref{obs prob lin}).

\begin{theorem}[Main Theorem 2]
\label{lin-obs-thm}
Let $2<p<\infty$ and let $w \in A_{\frac{p}{2}}$.
Suppose that $\partial \Omega \in C^{1,1}$, $f \in L^p_w(\Omega)$ and $\psi \in W^{2,p}_w(\Omega)$.
Then there exists a small $\delta=\delta(n,\lambda,\Lambda,p,w,\partial \Omega)>0$ such that if $\mathbf{A}$ is uniformly elliptic and $(\delta,R)$-vanishing, then there is a unique solution $u \in W^{2,p}_w(\Omega)$ to the obstacle problem (\ref{obs prob lin}) with the estimate
\begin{equation}
\Norm{u}_{W^{2,p}_w(\Omega)} \leq c \left( \Norm{f}_{L^p_w(\Omega)} + \Norm{\psi}_{W^{2,p}_w(\Omega)} \right),
\end{equation}
for some positive constant $c=c(n,\lambda,\Lambda,p,w,\partial \Omega,\mathrm{diam}(\Omega))$.
\end{theorem}

\begin{proof}
Since $\partial \Omega \in C^{1,1}$, there exists an extension $\overline{\psi}$ of $\psi$ to $\mr^n$ with $\overline{\psi}=\psi$ a.e. in $\Omega$, and
\begin{equation}
\Norm{\overline{\psi}}_{W^{2,p}_w(\mr^n)} \leq c \Norm{\psi}_{W^{2,p}_w(\Omega)},
\end{equation}
for some constant $c$ depending only on $n,p,w,\partial \Omega$ and $\mathrm{diam}(\Omega)$, see \cite{Ch}.
Let $g=f-a_{ij}D_{ij}\overline{\psi}$ in $\mr^n$ (we extend $f$ to zero outside $\Omega$).
Since $f \in L^p_w(\mr^n)$ and $\overline{\psi} \in W^{2,p}_w(\mr^n)$, we have $g \in L^p_w(\mr^n)$ with the estimate
\begin{align*}
\Norm{g}_{L^p_w(\mr^n)} & \leq c \left( \Norm{f}_{L^p_w(\mr^n)} + \Norm{\overline{\psi}}_{W^{2,p}_w(\mr^n)} \right) \\
& \leq c \left( \Norm{f}_{L^p_w(\Omega)} + \Norm{\psi}_{W^{2,p}_w(\Omega)} \right).
\end{align*}
Now let $\varphi$ denote a standard mollifier with support in $B_1$, and set $\varphi_{\varepsilon}(x):=\varepsilon^{-n}\varphi(x/\varepsilon)$.
We define the usual regularizations $a_{ij}^{\varepsilon} := a_{ij} \ast \varphi_{\varepsilon}$, $\overline{\psi}_{\varepsilon} := \overline{\psi} \ast \varphi_{\varepsilon}$, $f_{\varepsilon} := f \ast \varphi_{\varepsilon}$ and $g_{\varepsilon} := f_{\varepsilon} - a_{ij}^{\varepsilon} D_{ij}\overline{\psi}_{\varepsilon}$.
We note that for each $\varepsilon>0$, the matrix $(a_{ij}^{\varepsilon}) : \mr^n \rightarrow \mr^{n \times n}$ is uniformly elliptic with the same ellipticity constants.
Furthermore, $g_{\varepsilon} \to g$ almost everywhere, as $\varepsilon \to 0$, and
\begin{align*}
\Norm{g_{\varepsilon}}_{L^p_w(\mr^n)} & \leq c \left( \Norm{f_{\varepsilon}}_{L^p_w(\mr^n)} + \Norm{\overline{\psi}_{\varepsilon}}_{W^{2,p}_w(\mr^n)} \right) \\
& \leq c \left( \Norm{f}_{L^p_w(\mr^n)} + \Norm{\overline{\psi}}_{W^{2,p}_w(\mr^n)} \right) \\
& \leq c \left( \Norm{f}_{L^p_w(\Omega)} + \Norm{\psi}_{W^{2,p}_w(\Omega)} \right).
\end{align*}
Let $\Phi_{\varepsilon}(s)$ be the function in the proof of Theorem \ref{fully-obs-thm}.
We then consider the problem:
\begin{equation}
\label{lin approx}
\left\{\begin{array}{rclcc}\
a_{ij}^{\varepsilon} D_{ij}u_{\varepsilon} & = & g_{\varepsilon}^+ \Phi_{\varepsilon}(u_{\varepsilon}-\overline{\psi}_{\varepsilon}) + f_{\varepsilon} - g_{\varepsilon}^+ & \mathrm{in} & \Omega,\\
u_{\varepsilon} & = & 0 & \mathrm{on} & \partial \Omega.
\end{array}\right.
\end{equation}
According to Lemma \ref{lin eq}, for each $v_0 \in L^p_w(\Omega)$ there is $v \in W^{2,p}_w(\Omega)$ for which
\begin{equation*}
\left\{\begin{array}{rclcc}\
a_{ij}^{\varepsilon} D_{ij}v & = & g_{\varepsilon}^+ \Phi_{\varepsilon}(v_0-\overline{\psi}_{\varepsilon}) + f_{\varepsilon} - g_{\varepsilon}^+ & \mathrm{in} & \Omega,\\
v & = & 0 & \mathrm{on} & \partial \Omega.
\end{array}\right.
\end{equation*}
By the fact that $0 \leq \Phi_{\varepsilon}(\cdot) \leq 1$ and Lemma \ref{lin eq}, we find that
\begin{equation*}
\Norm{v}_{W^{2,p}_w(\Omega)} \leq R,
\end{equation*}
where $R$ is independent of $v_0$.
We set $v=Sv_0$.
Then we see that $S$ maps the $R$-ball in $L^p_w(\Omega)$ into itself and $S$ is compact.
It follows from Schauder's fixed point theorem that there is the unique $u_{\varepsilon}$ such that $u_{\varepsilon}=Su_{\varepsilon}$, which is the solution to the problem (\ref{lin approx}).
Lemma \ref{lin eq} now yields
\begin{align*}
\Norm{u_{\varepsilon}}_{W^{2,p}_w(\Omega)} & \leq c \left( \Norm{g_{\varepsilon}^+ \Phi_{\varepsilon}(u_{\varepsilon}-\overline{\psi}_{\varepsilon})}_{L^p_w(\Omega)} + \Norm{f_{\varepsilon}}_{L^p_w(\Omega)} + \Norm{g_{\varepsilon}^+}_{L^p_w(\Omega)} \right) \\
& \leq c \left( \Norm{g_{\varepsilon}^+}_{L^p_w(\Omega)} + \Norm{f_{\varepsilon}}_{L^p_w(\Omega)} + \Norm{g_{\varepsilon}^+}_{L^p_w(\Omega)} \right) \\
& \leq c \left( \Norm{f_{\varepsilon}}_{L^p_w(\Omega)} + \Norm{g_{\varepsilon}}_{L^p_w(\Omega)} \right) \\
& \leq c \left( \Norm{f}_{L^p_w(\mr^n)} + \Norm{g_{\varepsilon}}_{L^p_w(\mr^n)} \right) \\
& \leq c \left( \Norm{f}_{L^p_w(\Omega)} + \Norm{\psi}_{W^{2,p}_w(\Omega)} \right).
\end{align*}
Hence $\{ u_{\varepsilon} \}$ is uniformly bounded in $W^{2,p}_w(\Omega) \cap W^{1,2}_0(\Omega)$.
So we can find a subsequence $\{ u_{\varepsilon_k} \}_{k=1}^{\infty}$ with $\varepsilon_k \searrow 0$, and a function $u \in W^{2,p}_w(\Omega) \cap W^{1,2}_0(\Omega)$ such that $u_{\varepsilon_k}$ converges to $u$ weakly in $W^{2,p}_w(\Omega) \cap W^{1,2}_0(\Omega)$, and $u_{\varepsilon_k}$ converges to $u$ almost everywhere, as $\epsilon_k \rightarrow 0$.

We next claim that $u$ is a solution of the obstacle problem (\ref{obs prob lin}).
Since $u \in W^{1,2}_0(\Omega)$, $u=0$ on $\partial \Omega$.
It also follows from (\ref{lin approx}) that
\begin{equation*}
a_{ij}^{\varepsilon_k} D_{ij}u_{\varepsilon_k} = g_{\varepsilon_k}^+ \Phi_{\varepsilon_k}(u_{\varepsilon_k}-\overline{\psi}_{\varepsilon_k}) - g_{\varepsilon_k}^+ + f_{\varepsilon_k} \leq f_{\varepsilon_k} \ \quad \mathrm{in} \ \ \Omega,
\end{equation*}
for every $k$.
Passing to the limit $k \to \infty$, we obtain that $a_{ij}D_{ij}u \leq f$ a.e. in $\Omega$.

We now show that $u \geq \psi$ in $\Omega$.
To do this, fix $k \in \mathbb{N}$, and the note that
$\Phi_{\varepsilon_k}(u_{\varepsilon_k}-\overline{\psi}_{\varepsilon_k})=0$
on the set $V_k := \left\lbrace u_{\varepsilon_k} < \overline{\psi}_{\varepsilon_k} \right\rbrace$.
Hence, $a_{ij}^{\varepsilon_k} D_{ij}u_{\varepsilon_k} = f_{\varepsilon_k} - g_{\varepsilon_k}^+$ in $V_k$.
If $V_k \neq \emptyset$, then it follows from the definition of $g_{\varepsilon_k}$ that
\begin{equation*}
a_{ij}^{\varepsilon_k} D_{ij}u_{\varepsilon_k} = f_{\varepsilon_k} - g_{\varepsilon_k}^+ = f_{\varepsilon_k} - g_{\varepsilon_k} - g_{\varepsilon_k}^- = a_{ij}^{\varepsilon_k} D_{ij}\overline{\psi}_{\varepsilon_k} - g_{\varepsilon_k}^- \leq a_{ij}^{\varepsilon_k} D_{ij}\overline{\psi}_{\varepsilon_k} \ \quad \mathrm{in} \ \ V_k.
\end{equation*}
Since $u_{\varepsilon_k}=\overline{\psi}_{\varepsilon_k}$ on $\partial V_k$, we discover that
\begin{equation}
\left\{\begin{array}{rclcc}\
a_{ij}^{\varepsilon_k} D_{ij}(u_{\varepsilon_k} - \overline{\psi}_{\varepsilon_k}) & \leq & 0 & \mathrm{in} & V_k,\\
u_{\varepsilon_k} - \overline{\psi}_{\varepsilon_k} & \geq & 0 & \mathrm{on} & \partial V_k.
\end{array}\right.
\end{equation}
Then in light of Lemma \ref{comparison-lin}, $u_{\varepsilon_k} - \overline{\psi}_{\varepsilon_k} \geq 0$ in $V_k$, which contradicts the definition of the set $V_k$, and we conclude that $V_k = \emptyset$ and $u_{\varepsilon_k} \geq \overline{\psi}_{\varepsilon_k}$ in $\Omega$.
But then since $k \in \mathbb{N}$ is arbitrary, passing to the limit $k \to \infty$, we discover that $u \geq \overline{\psi}$ a.e. in $\Omega$.
Therefore, $u \geq \psi$ a.e. in $\Omega$, as $\overline{\psi}=\psi$ a.e. in $\Omega$.

We next show that $a_{ij}D_{ij}u = f$ on the set $\left\lbrace u > \psi \right\rbrace$.
Observe that for each $m \in \mathbb{N}$, $\Phi_{\varepsilon_k}(u_{\varepsilon_k}-\overline{\psi}_{\varepsilon_k})$ converges to $1$ almost everywhere on the set $\left\lbrace u > \overline{\psi} + \frac{1}{m} \right\rbrace$, to find
\begin{equation*}
a_{ij}D_{ij}u = g^+ + f - g^+ = f
\end{equation*}
on the set $\left\lbrace u > \psi \right\rbrace = \left\lbrace u > \overline{\psi} \right\rbrace = \bigcup_{m=1}^{\infty} \left\lbrace u > \overline{\psi} + \frac{1}{m} \right\rbrace$.
Consequently, $u \in W^{2,p}_w(\Omega)$ is a solution to (\ref{obs prob fully}) with the estimate
\begin{equation}
\Norm{u}_{W^{2,p}_w(\Omega)} \leq \liminf_{k \to \infty} \Norm{u_{\varepsilon_k}}_{W^{2,p}_w(\Omega)} \leq c \left( \Norm{f}_{L^p_w(\Omega)} + \Norm{\psi}_{W^{2,p}_w(\Omega)} \right).
\end{equation}

Now it remains to prove the uniqueness.
Let $u_1, u_2$ be solutions to (\ref{obs prob lin}) and assume that the open set $G := \{ u_2 > u_1 \}$ is nonempty.
Since $u_2 > u_1 \geq \psi$ in $G$, we know that $a_{ij}D_{ij}u_2 = f$ in $G$.
Therefore, we have
\begin{equation*}
a_{ij}D_{ij}(u_2-u_1) = f - a_{ij}D_{ij}u_1 \geq 0 \ \quad \mathrm{in} \ \ G
\end{equation*}
and
\begin{equation*}
u_2 = u_1 \ \quad \mathrm{on} \ \ \partial G.
\end{equation*}
Then Lemma \ref{comparison-lin} implies that $u_2-u_1 \leq 0$ in $G$, which is a contradiction, and therefore $u_1=u_2$.
This finishes the proof.
\end{proof}

\section{\bf Morrey regularity results and H\"{o}lder continuity of the gradient}
\label{sec5}

The \emph{Morrey space} $L^{p,\theta}(\Omega)$ with $p \in [1,\infty)$ and $\theta \in [0,n]$ consists of all measurable functions $g \in L^p(\Omega)$ for which the norm
\begin{equation*}
\Norm{g}_{L^{p,\theta}(\Omega)}:=\left( \sup_{y \in \Omega, r>0} \frac{1}{r^{\theta}}\int_{B_r(y) \cap \Omega} |g(x)|^p \ dx \right)^{\frac{1}{p}}
\end{equation*}
is finite.
The \emph{Sobolev-Morrey space} $W^{2,p,\theta}(\Omega)$ consists of all functions $g \in W^{2,p}(\Omega)$ such that the second order derivatives belongs to the Morrey space $L^{p,\theta}(\Omega)$.
A natural norm of this space is defined by
\begin{equation*}
\Norm{g}_{W^{2,p,\theta}(\Omega)} := \Norm{g}_{L^p(\Omega)}+\Norm{D^2 g}_{L^{p,\theta}(\Omega)}.
\end{equation*}
We note that for $p \in [1,\infty)$, $L^{p,0}(\Omega) \cong L^p(\Omega)$ and $L^{p,n}(\Omega) \cong L^{\infty}(\Omega)$.
Hence, we deal with only the case $0 < \theta < n$ in this section.

We now state and prove the main results in this section, the Morrey regularity results for the elliptic obstacle problems (\ref{obs prob fully}) and (\ref{obs prob lin}).

\begin{theorem}
\label{fully-obs-morrey}
Let $n_0 < p < \infty$, where $n_0:=n-\nu_0$ for some $\nu_0 = \nu_0 \left( \frac{\Lambda}{\lambda},n \right) > 0$, and let $0 < \theta < n$.
Suppose that $\partial \Omega \in C^{1,1}$, $f \in L^{p,\theta}(\Omega)$ and $\psi \in W^{2,p,\theta}(\Omega)$.
Then there exists a small $\delta=\delta(n,\lambda,\Lambda,p,\theta,\partial \Omega)>0$ such that if (\ref{con beta}) is satisfied for some $R_0 > 0$, then the fully nonlinear obstacle problem (\ref{obs prob fully}) has a unique solution $u \in W^{2,p,\theta}(\Omega)$ with the estimate
\begin{equation}
\label{5-1}
\Norm{u}_{W^{2,p,\theta}(\Omega)} \leq c \left( \Norm{f}_{L^{p,\theta}(\Omega)} + \Norm{\psi}_{W^{2,p,\theta}(\Omega)} \right),
\end{equation}
for some positive constant $c=c(n,\lambda,\Lambda,p,\theta,\partial \Omega,\mathrm{diam}(\Omega),R_0)$.
\end{theorem}

\begin{theorem}
\label{lin-obs-morrey}
Let $2<p<\infty$ and let $0 < \theta < n$.
Suppose that $\partial \Omega \in C^{1,1}$, $f \in L^{p,\theta}(\Omega)$ and $\psi \in W^{2,p,\theta}(\Omega)$.
There exists a small $\delta=\delta(n,\lambda,\Lambda,p,\theta,\partial \Omega)>0$ such that if $\mathbf{A}$ is uniformly elliptic and $(\delta,R)$-vanishing, then the obstacle problem (\ref{obs prob lin}) has a unique solution $u \in W^{2,p,\theta}(\Omega)$ and we have the estimate
\begin{equation}
\label{5-2}
\Norm{u}_{W^{2,p,\theta}(\Omega)} \leq c \left( \Norm{f}_{L^{p,\theta}(\Omega)} + \Norm{\psi}_{W^{2,p,\theta}(\Omega)} \right),
\end{equation}
for some positive constant $c=c(n,\lambda,\Lambda,p,\theta,\partial \Omega,\mathrm{diam}(\Omega))$.
\end{theorem}

\begin{proof}[Proof of Theorem \ref{fully-obs-morrey} and \ref{lin-obs-morrey}]
Throughout the proof, we use the number $m_0$ to denote the number $n_0$ when proving Theorem \ref{fully-obs-morrey}, and the number $2$ when proving Theorem \ref{lin-obs-morrey}, respectively.

We first recall that for a locally integrable function $h : \mr^n \to \mr$, the Hardy-Littlewood maximal function of $h$ is defined by
\begin{equation*}
\mathcal{M}h(x) := \sup_{r>0} \frac{1}{|B_r(x)|} \int_{B_r(x)} |h(y)| \, dy,
\end{equation*}
for $x \in \mr^n$.
From \cite[Proposition 2]{CR}, we see that if $\sigma \in (0,1)$, then
\begin{equation*}
\left( \mathcal{M}\chi_{B_r(x_0)}(x) \right)^{\sigma} \in A_1,
\end{equation*}
where $\chi_{B_r(x_0)}$ is the characteristic function of $B_r(x_0)$.
Hence, it follows from the fact that $p>m_0$ and the monotonicity of the classes $A_s$ that
\begin{equation*}
\left( \mathcal{M}\chi_{B_r(x_0)}(x) \right)^{\sigma} \in A_1 \subset A_{\frac{p}{m_0}},
\end{equation*}
with $\left[ \left( \mathcal{M}\chi_{B_r(x_0)}(x) \right)^{\sigma} \right]_{\frac{p}{m_0}} \leq c(n,m_0,p,\sigma)$.

We now fix any $\sigma \in \left( \frac{\theta}{n},1 \right)$.
Then by Theorem \ref{fully-obs-thm} and \ref{lin-obs-thm}, we have
\begin{align*}
& \int_{B_r(x_0) \cap \Omega} |D^2 u|^p \ dx = \int_{\Omega} |D^2 u|^p \left( \chi_{B_r(x_0)} \right)^{\sigma} dx \\
& \quad \leq \int_{\Omega} |D^2 u|^p \left( \mathcal{M} \chi_{B_r(x_0)} \right)^{\sigma} dx \\
& \quad \leq c \left( \int_{\Omega} |f|^p \left( \mathcal{M}
\chi_{B_r(x_0)} \right)^{\sigma} dx + \int_{\Omega} |D^2 \psi|^p
\left( \mathcal{M} \chi_{B_r(x_0)} \right)^{\sigma} dx \right),
\end{align*}
for some positive constant $c$ depending only on $n,\lambda,\Lambda,p,\theta,\partial \Omega$ and $\mathrm{diam}(\Omega)$.

We next use the following set decomposition
\begin{equation*}
\Omega = \left( B_{2r}(x_0) \cap \Omega \right) \cup \left( \bigcup_{k=1}^{\infty}
\left( B_{2^{k+1}r}(x_0) \setminus B_{2^{k}r}(x_0) \right) \cap \Omega \right),
\end{equation*}
to find
\begin{multline}
\label{5-4}
\int_{\Omega} |f|^p \left( \mathcal{M} \chi_{B_r(x_0)} \right)^{\sigma} dx = \int_{B_{2r}(x_0) \cap \Omega} |f|^p \left( \mathcal{M} \chi_{B_r(x_0)} \right)^{\sigma} dx \\
+ \sum_{k=1}^{\infty} \int_{(B_{2^{k+1}r}(x_0) \setminus B_{2^{k}r}(x_0)) \cap \Omega} |f|^p \left( \mathcal{M} \chi_{B_r(x_0)} \right)^{\sigma} dx.
\end{multline}
Since $\mathcal{M} \chi_{B_r(x_0)} \leq 1$, we have the
estimate
\begin{equation}
\label{5-5}
\int_{B_{2r}(x_0) \cap \Omega} |f|^p \left( \mathcal{M} \chi_{B_r(x_0)} \right)^{\sigma} dx \leq \int_{B_{2r}(x_0) \cap \Omega} |f|^p \, dx \leq r^{\theta} \Norm{f}_{L^{p,\theta}(\Omega)}^p.
\end{equation}
We note that for each $x \in B_{2^{k+1}r}(x_0) \setminus
B_{2^{k}r}(x_0)$ and for each $\rho > (2^{k+1}-1)r$,
\begin{equation*}
0 < \mint_{B_{\rho}(x)} \chi_{B_r(x_0)}(y) \, dy \leq \frac{|B_r(x_0)|}{|B_{\rho}(x)|} = \left( \frac{r}{\rho} \right)^n.
\end{equation*}
Since $2^{k+1}-1 \geq 2^{k}-1 \geq 2^{k-1}$, it follows that
\begin{equation*}
\mint_{B_{\rho}(x)} \chi_{B_r(x_0)}(y) \, dy \leq \left( \frac{r}{2^{k-1}r} \right)^n = \frac{1}{2^{n(k-1)}},
\end{equation*}
and hence
\begin{equation*}
\left( \mathcal{M} \chi_{B_r(x_0)}(x) \right)^{\sigma} = \left( \sup_{\rho>0} \mint_{B_{\rho}(x)} \chi_{B_r(x_0)}(y) \, dy \right)^{\sigma} \leq \frac{1}{2^{\sigma n(k-1)}}.
\end{equation*}
Therefore, we deduce that for each $k=1,2,\cdots$,
\begin{align}
\nonumber & \int_{(B_{2^{k+1}r}(x_0) \setminus B_{2^{k}r}(x_0)) \cap \Omega} |f|^p \left( \mathcal{M} \chi_{B_r(x_0)} \right)^{\sigma} dx \\
\nonumber & \qquad \leq \frac{1}{2^{\sigma n(k-1)}} \int_{(B_{2^{k+1}r}(x_0) \setminus B_{2^{k}r}(x_0)) \cap \Omega} |f|^p \, dx \\
\nonumber & \qquad \leq \frac{1}{2^{\sigma n(k-1)}} \int_{B_{2^{k+1}r}(x_0) \cap \Omega} |f|^p \, dx \\
\label{5-6} & \qquad \leq c(n) \frac{(2^{k+1}r)^{\theta}}{2^{\sigma n(k-1)}} \Norm{f}_{L^{p,\theta}(\Omega)}^p = c(n) 2^{(\sigma n + \theta) - (\sigma n - \theta)k} r^{\theta} \Norm{f}_{L^{p,\theta}(\Omega)}^p.
\end{align}
We combine (\ref{5-5}) and (\ref{5-6}) with (\ref{5-4}) to derive
\begin{align*}
\int_{\Omega} |f|^p \left( \mathcal{M} \chi_{B_r(x_0)} \right)^{\sigma} dx & \leq c(n) r^{\theta} \left( 1 + 2^{\sigma n + \theta} \sum_{k=1}^{\infty} 2^{-(\sigma n - \theta)k} \right) \Norm{f}_{L^{p,\theta}(\Omega)}^p \\
& \leq c r^{\theta} \left( \sum_{k=0}^{\infty} 2^{-(\sigma n - \theta)k} \right) \Norm{f}_{L^{p,\theta}(\Omega)}^p \\
& \leq c r^{\theta} \Norm{f}_{L^{p,\theta}(\Omega)}^p.
\end{align*}
Similarly, we find
\begin{equation*} \int_{\Omega} |D^2 \psi|^p
\left( \mathcal{M} \chi_{B_r(x_0)} \right)^{\sigma} dx \leq c
r^{\theta} \Norm{D^2 \psi}_{L^{p,\theta}(\Omega)}^p.
\end{equation*}
Thus
\begin{equation*}
\int_{B_r(x_0) \cap \Omega} |D^2 u|^p \, dx \leq c r^{\theta} \left( \Norm{f}_{L^{p,\theta}(\Omega)}^p + \Norm{D^2 \psi}_{L^{p,\theta}(\Omega)}^p \right).
\end{equation*}
Dividing the both sides by $r^{\theta}$ and taking the supremum with respect to $x_0 \in \Omega$ and $r>0$, we conclude that $D^2 u \in L^{p,\theta}(\Omega)$ with the desired estimates (\ref{5-1}) and (\ref{5-2}).
This completes the proof.
\end{proof}

For the fully nonlinear obstacle problem (\ref{obs prob fully}), we have H\"{o}lder continuity of the gradient of the solution when $p>n$, by the Sobolev embedding theorem.
However, for the linear obstacle problem (\ref{obs prob lin}), we cannot obtain such a result directly, when $2 < p \leq n$.
Nevertheless, the Morrey regularity result (see Theorem \ref{lin-obs-morrey}) and the following Sobolev-Morrey embedding lemma allow to prove H\"{o}lder continuity of the gradient of the solution for appropriate values of $p$ and $\theta$.

\begin{lemma}
\label{morrey embedding}
\cite[Lemma 3.III and Lemma 3.IV]{Cam}
Suppose that $\Omega$ is a bounded domain with $\partial \Omega \in C^{1,1}$.
Let $v \in W^{1,p,\theta}(\Omega)$.
If $p + \theta > n$, then $v \in C^{0,\alpha}(\overline{\Omega})$ for $\alpha = 1 - \frac{n-\theta}{p}$ and we have the estimate
\begin{equation*}
\Norm{v}_{C^{0,\alpha}(\overline{\Omega})} \leq c \Norm{v}_{W^{1,p,\theta}(\Omega)},
\end{equation*}
where $c$ is a positive constant depending only on $n,p,\theta$ and $\partial \Omega$.
\end{lemma}

We now state the H\"{o}lder continuity result of the gradient of the solution to the linear obstacle problem (\ref{obs prob lin}).

\begin{theorem}
Under the assumptions of Theorem \ref{lin-obs-morrey}, let $u \in W^{2,p,\theta}(\Omega)$ be the solution to the obstacle problem (\ref{obs prob lin}).
If $p + \theta > n$, then $Du \in C^{0,\alpha}(\overline{\Omega})$ for $\alpha = 1-\frac{n-\theta}{p}$ and we have the estimate
\begin{equation}
\Norm{Du}_{C^{0,\alpha}(\overline{\Omega})} \leq c \left( \Norm{f}_{L^{p,\theta}(\Omega)} + \Norm{\psi}_{W^{2,p,\theta}(\Omega)} \right),
\end{equation}
for some positive constant $c=c(n,\lambda,\Lambda,p,\theta,\partial \Omega,\mathrm{diam}(\Omega))$.
\end{theorem}

\begin{proof}
The proof follows directly from Theorem \ref{lin-obs-morrey} and Lemma \ref{morrey embedding}.
\end{proof}

\section{\bf Parabolic obstacle problem}
\label{sec6}

In this section we consider the parabolic obstacle problem (\ref{parabolic obs prob}).
As in the elliptic case, we first introduce the small BMO assumption on the coefficient matrix $\mathbf{A}(x,t)$.

\begin{definition}
We say that the coefficient matrix $\mathbf{A}=\mathbf{A}(x,t)$ is \textit{$(\delta,R)$-vanishing} if
\begin{equation}
\sup_{0 < r \leq R} \ \sup_{(y,s) \in \mr^n \times \mr} \left(
\mint_{Q_r(y,s)}
|\mathbf{A}(x,t)-\overline{\mathbf{A}}_{Q_r(y,s)}|^2 \, dx dt
\right)^{\frac{1}{2}} \leq \delta,
\end{equation}
where $\overline{\mathbf{A}}_{Q_r(y,s)} = \mint_{Q_r(y,s)} \mathbf{A}(x,t) \, dx dt$ is the integral average of $\mathbf{A}(x,t)$ on the parabolic cylinder $Q_r(y,s)$.
\end{definition}

We remark that one can take $R=1$ as in the elliptic case, which is due to the scaling invariance property.
On the other hand, $\delta>0$ is invariant under such a scaling.

We now provide the definitions of the Muckenhoupt classes and weighted Sobolev space in the parabolic version.
For a given $1<s<\infty$, we say that $w$ is a \textit{weight} in \textit{Muckenhoupt class $A_s$}, or an \textit{$A_s$ weight}, if $w$ is a locally integrable nonnegative function on $\mr^{n+1}$ with
\begin{equation}
[w]_s := \sup_{Q} \left( \mint_{Q} w(x,t) \, dx dt \right) \left( \mint_{Q} w(x,t)^{-\frac{1}{s-1}} \ dx dt \right)^{s-1} < +\infty,
\end{equation}
where the supremum is taken over all parabolic cylinders $Q \subset \mr^{n+1}$.
If $w$ is an $A_s$ weight, we write $w \in A_s$, and $[w]_s$ is called the \textit{$A_s$ constant of $w$}.

The \emph{weighted Lebesgue space} $L^p_w(\Omega_T)$, $1<p<\infty$, $w \in A_s$ with $1<s<\infty$, consists of all measurable functions $g=g(x,t)$ on $\Omega_T$ such that
\begin{equation*}
\Norm{g}_{L^p_w(\Omega_T)}:=\left( \int_{\Omega_T} |g(x,t)|^p w(x,t) \ dx dt \right)^{\frac{1}{p}} < +\infty.
\end{equation*}

The \emph{weighted Sobolev space} $W^{2,1}L^p_w(\Omega_T)$, $1<p<\infty$, $w\in A_s$ with $1<s<\infty$, is defined by a class of functions $g=g(x,t) \in L^p_w(\Omega_T)$ with distributional derivatives $D_t^r D_x^\alpha g(x,t) \in L^p_w(\Omega_T)$ for $0
\leq 2r+|\alpha| \leq 2$.
The norm of $g$ in $W^{2,1}L^p_w(\Omega_T)$ is defined by
\begin{equation*}
\Norm{g}_{W^{2,1}L^p_w(\Omega_T)}:=\left( \sum_{j=0}^{2} \sum_{2r+|\alpha|=j} \int_{\Omega_T} |D_t^r D_x^\alpha g(x,t)|^p w(x,t) \ dx dt \right)^{\frac{1}{p}}.
\end{equation*}
In addition, let $W_0^{2,1}L^p_w(\Omega_T)$ be the closure in the $W^{2,1}L^p_w(\Omega_T)$ norm of the space
\begin{equation*}
\mathcal{C} = \left\lbrace \phi \in C^{\infty}(\Omega_T) : \phi(x,t)=0 \ \ \mathrm{for} \ \ (x,t) \in \partial_p \Omega \right\rbrace.
\end{equation*}

We will utilize the following weighted $W^{2,1}L^p$ estimate for linear parabolic equations without obstacle, see \cite{BL1}.

\begin{lemma}
\label{parabolic eq}
Let $2<p<\infty$ and let $w=w(x,t) \in A_{\frac{p}{2}}$.
Suppose that $\partial \Omega \in C^{1,1}$ and $f \in L^p_w(\Omega_T)$.
Then there exists a small $\delta=\delta(\Lambda,p,n,w,\partial \Omega,T)>0$ such that if $\mathbf{A}$ is uniformly parabolic and $(\delta,R)$-vanishing, then the following problem
\begin{equation*}
\left\{ \begin{array}{rclcc}
u_t - a_{ij}D_{ij}u & = & f & \mathrm{in} & \Omega_T,\\
u & = & 0 & \mathrm{on} & \partial_p \Omega_T
\end{array}\right.
\end{equation*}
has a unique solution $u \in W^{2,1}L^p_w(\Omega_T)$ with the estimate
\begin{equation}
\Norm{u}_{W^{2,1}L^p_w(\Omega_T)} \le c \Norm{f}_{L^p_w(\Omega_T)}
\end{equation}
for some positive constant $c=c(n,\lambda,\Lambda,p,w,\partial \Omega,\mathrm{diam}(\Omega),T)$.
\end{lemma}

For an open set $U \subset \mr^{n+1}$, $C^{2,1}(U)$ ($C^{2,1}(\overline{U})$) is defined by a set of continuous functions in $U$ (in $\overline{U}$) having continuous derivatives $D_x u, D^2_x u, D_t u$ in $U$ (in $\overline{U}$).
We also define the parabolic boundary $\partial_p U$ to be the set of all points $(x,t) \in \partial U$ such that for any $r>0$, the parabolic cylinder $Q_r(x,t)$ contains points not in $U$.
We remark that in the special case $U = \Omega_T = \Omega \times (0,T]$, the parabolic boundary $\partial_p U$ of $U$ coincides with $\partial_p \Omega_T = (\partial \Omega \times [0,T]) \cup (\Omega \times \{t=0\})$.

The following maximum principle for linear parabolic equations can be found in \cite[Lemma 2.1]{Li}.

\begin{lemma}
\label{comparison-parabolic}
Let $U \subset \Omega_T$ be a bounded domain.
Suppose that $\mathbf{A}=(a_{ij})$ is uniformly parabolic with $a_{ij} \in C(\mr^{n+1})$.
If $u \in C^{2,1}(\overline{U})$ satisfies
\begin{equation*}
\left\{\begin{array}{rclcc}\
u_t - a_{ij}D_{ij}u & \geq & 0 & \mathrm{in} & U,\\
u & \geq & 0 & \mathrm{on} & \partial_p U,
\end{array}\right.
\end{equation*}
then $u \geq 0$ in $U$.
\end{lemma}

Let us now state and prove the last main result in this paper regarding the parabolic obstacle problem (\ref{parabolic obs prob}).

\begin{theorem}[Main Theorem 3]
\label{parabolic-obs-thm}
Let $2<p<\infty$ and let $w=w(x,t) \in A_{\frac{p}{2}}$.
Suppose that $\partial \Omega \in C^{1,1}$, $f \in L^p_w(\Omega_T)$ and $\psi \in W^{2,1}L^p_w(\Omega_T)$.
There exists a small $\delta=\delta(n,\lambda,\Lambda,p,w,\partial \Omega,T)>0$ such that if $\mathbf{A}$ is uniformly parabolic and $(\delta,R)$-vanishing, then the obstacle problem (\ref{parabolic obs prob}) has a solution $u \in W^{2,1}L^p_w(\Omega_T)$ and we have the estimate
\begin{equation}
\Norm{u}_{W^{2,1}L^p_w(\Omega_T)} \leq c \left( \Norm{f}_{L^p_w(\Omega_T)}
+ \Norm{\psi}_{W^{2,1}L^p_w(\Omega_T)} \right),
\end{equation}
for some positive constant $c=c(n,\lambda,\Lambda,p,w,\partial \Omega,\mathrm{diam}(\Omega),T)$.
\end{theorem}

\begin{proof}[Proof of Theorem \ref{parabolic-obs-thm}]
We first note that since $\partial \Omega \in C^{1,1}$, there exists an extension $\overline{\psi}$ of $\psi$ to $\mr^{n+1}$ with $\overline{\psi}=\psi$ a.e. in $\Omega_T$, and
\begin{equation}
\Norm{\overline{\psi}}_{W^{2,1}L^p_w(\mr^{n+1})} \leq c \Norm{\psi}_{W^{2,1}L^p_w(\Omega_T)},
\end{equation}
for some constant $c=c(n,p,w,\partial \Omega,\mathrm{diam}(\Omega),T)$, see \cite{Ch}.
Let $g = -f + \overline{\psi}_t - a_{ij}D_{ij}\overline{\psi}$ in $\mr^{n+1}$ (we extend $f$ to zero outside $\Omega_T$).
Define $f \in L^p_w(\mr^{n+1})$ and $\overline{\psi} \in W^{2,1}L^p_w(\mr^{n+1})$.
Then we see that $g \in L^p_w(\mr^{n+1})$ with the estimate
\begin{align*}
\Norm{g}_{L^p_w(\mr^{n+1})} & \leq c \left( \Norm{f}_{L^p_w(\mr^{n+1})} + \Norm{\overline{\psi}}_{W^{2,1}L^p_w(\mr^{n+1})} \right) \\
& \leq c \left( \Norm{f}_{L^p_w(\Omega_T)} + \Norm{\psi}_{W^{2,1}L^p_w(\Omega_T)} \right).
\end{align*}
We now let $\varphi$ denote a standard mollifier with support in
$Q_1$, and define $\varphi_{\varepsilon}(x,t):=\varepsilon^{-(n+1)}\varphi((x,t)/\varepsilon)$.
We then consider the regularizations $a_{ij}^{\varepsilon}:= a_{ij} \ast \varphi_{\varepsilon}$, $f_{\varepsilon}:= f \ast \varphi_{\varepsilon}$, $\overline{\psi}_{\varepsilon}:= \overline{\psi} \ast \varphi_{\varepsilon}$,
$(\overline{\psi}_t)_{\varepsilon}:= \overline{\psi}_t \ast \varphi_{\varepsilon} = (\overline{\psi}_{\varepsilon})_t$ and $g_{\varepsilon}: = -f_{\varepsilon} +(\overline{\psi}_{\varepsilon})_t - a_{ij}^{\varepsilon} D_{ij}\overline{\psi}_{\varepsilon}$.
We note that for each $\varepsilon>0$, the matrix $(a_{ij}^{\varepsilon}): \mr^n \times \mr \rightarrow \mr^{n \times n}$ is uniformly parabolic with the same constants $\lambda$ and $\Lambda$.
Moreover, we see that $g_{\varepsilon} \to g$ almost everywhere, as $\varepsilon \to 0$, and that
\begin{align*}
\Norm{g_{\varepsilon}}_{L^p_w(\mr^{n+1})} & \leq c \left( \Norm{f_{\varepsilon}}_{L^p_w(\mr^{n+1})} + \Norm{\overline{\psi}_{\varepsilon}}_{W^{2,1}L^p_w(\mr^{n+1})} \right) \\
& \leq c \left( \Norm{f}_{L^p_w(\mr^{n+1})} + \Norm{\overline{\psi}}_{W^{2,1}L^p_w(\mr^{n+1})} \right) \\
& \leq c \left( \Norm{f}_{L^p_w(\Omega_T)} + \Norm{\psi}_{W^{2,1}L^p_w(\Omega_T)} \right).
\end{align*}

We next let $\Phi_{\varepsilon}(s)$ be the function in the proof of Theorem \ref{fully-obs-thm}, and define $\Psi_{\varepsilon}(s):=s\Phi_{\varepsilon}(s)$ for $s \in \mr$. Then the function $\Psi_{\varepsilon} \in C^{\infty}(\mr)$ is non-decreasing and satisfies
\begin{equation*}
\Psi_{\varepsilon}(s) \equiv 0 \quad \mathrm{if} \quad s \leq 0; \qquad \Psi_{\varepsilon}(s) \equiv s \quad \mathrm{if} \quad s \geq \varepsilon,
\end{equation*}
and
\begin{equation*}
0 \leq \Psi_{\varepsilon}(s) \leq s, \quad \forall s \geq 0.
\end{equation*}
Now let us look at the following problem:
\begin{equation}
\label{lin approx-parabolic}
\left\{\begin{array}{rclc}\
(u_{\varepsilon})_t - a_{ij}^{\varepsilon} D_{ij}u_{\varepsilon} & = & - \Psi_{\varepsilon}(g_{\varepsilon}) \, \Phi_{\varepsilon}(u_{\varepsilon}-\overline{\psi}_{\varepsilon}) + \Psi_{\varepsilon}(g_{\varepsilon}) + f_{\varepsilon} &  \hspace{-0.4cm} \mathrm{ in } \  \Omega_T,\\
u_{\varepsilon} & = & 0 & \mathrm{ on } \  \ \partial_p \Omega_T.
\end{array}\right.
\end{equation}
According to Lemma \ref{parabolic eq}, we find that for each $v_0 \in L^p_w(\Omega_T)$, there exists a function $v \in W^{2,1}L^p_w(\Omega_T)$ such that
\begin{equation*}
\left\{\begin{array}{rclcc}\
v_t - a_{ij}^{\varepsilon} D_{ij}v & = & - \Psi_{\varepsilon}(g_{\varepsilon}) \ \Phi_{\varepsilon}(v_0-\overline{\psi}_{\varepsilon}) + \Psi_{\varepsilon}(g_{\varepsilon}) + f_{\varepsilon} & \mathrm{in} & \Omega_T,\\
v & = & 0 & \mathrm{on} & \partial_p \Omega_T.
\end{array}\right.
\end{equation*}
Recall that $0 \leq \Phi_{\varepsilon}(s) \leq 1$ and $0 \leq \Psi_{\varepsilon}(s) \leq |s|$ for all $s \in \mr$, to find from Lemma \ref{parabolic eq} that
\begin{equation*}
\Norm{v}_{W^{2,1}L^p_w(\Omega_T)} \leq R,
\end{equation*}
where $R$ is independent of $v_0$.
Let us write $v=Sv_0$.
Then observe that $S$ maps the $R$-ball in $L^p_w(\Omega_T)$ into itself and that $S$ is compact.
Thus it follows from Schauder's fixed point theorem that there is a unique $u_{\varepsilon}$ such that $u_{\varepsilon}=Su_{\varepsilon}$, which is the solution to the problem (\ref{lin approx-parabolic}).

Lemma \ref{parabolic eq} now yields
\begin{align*}
& \Norm{u_{\varepsilon}}_{W^{2,1}L^p_w(\Omega_T)} \\
& \quad \leq c \left( \Norm{\Psi_{\varepsilon}(g_{\varepsilon}) \ \Phi_{\varepsilon}(u_{\varepsilon}-\overline{\psi}_{\varepsilon})}_{L^p_w(\Omega_T)} + \Norm{\Psi_{\varepsilon}(g_{\varepsilon})}_{L^p_w(\Omega_T)} + \Norm{f_{\varepsilon}}_{L^p_w(\Omega_T)} \right) \\
& \quad \leq c \left( \Norm{g_{\varepsilon}}_{L^p_w(\Omega_T)} + \Norm{g_{\varepsilon}}_{L^p_w(\Omega_T)} + \Norm{f_{\varepsilon}}_{L^p_w(\Omega_T)} \right) \\
& \quad \leq c \left( \Norm{g_{\varepsilon}}_{L^p_w(\mr^{n+1})} + \Norm{f}_{L^p_w(\mr^{n+1})} \right) \\
& \quad \leq c \left( \Norm{f}_{L^p_w(\Omega_T)} + \Norm{\psi}_{W^{2,1}L^p_w(\Omega_T)} \right).
\end{align*}
Hence $\{ u_{\varepsilon} \}$ is uniformly bounded in $W_0^{2,1}L^p_w(\Omega_T)$.
So we can find a subsequence $\{ u_{\varepsilon_k} \}_{k=1}^{\infty}$ with $\varepsilon_k \searrow 0$, and a function $u \in W_0^{2,1}L^p_w(\Omega_T)$ such that $u_{\varepsilon_k}$ converges to $u$ weakly in $W^{2,1}L^p_w(\Omega_T)$, and $u_{\varepsilon_k}$ converges to $u$ almost everywhere.
Since $f_{\varepsilon}, g_{\varepsilon}, \overline{\psi}_{\varepsilon}$ and $\Psi_{\varepsilon}$ are smooth functions, Lemma \ref{parabolic eq} implies that $u_{\varepsilon} \in W_0^{2,1}L^q(\Omega_T)$ for all $q \in (2,\infty)$, and so $u_{\varepsilon} \in C^{\alpha}(\Omega_T)$ for some $\alpha \in (0,1)$.
Then, by Schauder's theorem, $u_{\varepsilon}, D_x u_{\varepsilon}, D^2_x u_{\varepsilon}$ and $D_t u_{\varepsilon}$ belong to $C^{\alpha}(\Omega_T)$, and so we conclude that $u_{\varepsilon} \in C^{2,1}(\Omega_T)$.

We next claim that $u$ is a solution of the obstacle problem (\ref{parabolic obs prob}).
Observe that $u=0$ on $\partial_p \Omega_T$.
We recall (\ref{lin approx-parabolic}) to discover that
\begin{equation*}
\left( u_{\varepsilon_k} \right)_t - a_{ij}^{\varepsilon_k} D_{ij}u_{\varepsilon_k} = \Psi_{\varepsilon_k}(g_{\varepsilon_k}) \left( 1 - \Phi_{\varepsilon_k} (u_{\varepsilon_k}-\overline{\psi}_{\varepsilon_k}) \right) + f_{\varepsilon_k} \geq f_{\varepsilon_k} \ \quad \mathrm{in} \ \ \Omega.
\end{equation*}
Passing to the limit $k \to \infty$, we find that $u_t - a_{ij}D_{ij}u \geq f$ a.e. in $\Omega_T$.

We next want to show that $u \geq \psi$ in $\Omega_T$.
To do this, fix $k \in \mathbb{N}$.
We then observe that $\Phi_{\varepsilon_k}(u_{\varepsilon_k}-\overline{\psi}_{\varepsilon_k})=0$ on $V_k := \left\lbrace u_{\varepsilon_k} < \overline{\psi}_{\varepsilon_k} \right\rbrace$, and so we discover that $(u_{\varepsilon_k})_t - a_{ij}^{\varepsilon_k} D_{ij}u_{\varepsilon_k} = \Psi_{\varepsilon_k}(g_{\varepsilon_k}) + f_{\varepsilon_k}$ in $V_k$.
If $V_k = \emptyset$, then $u_{\varepsilon_k} \geq \overline{\psi}_{\varepsilon_k}$ in $\Omega_T$.
On the other hand, if $V_k \neq \emptyset$, then it follows from the definition of $\Psi_{\varepsilon_k}$ and $g_{\varepsilon_k}$ that
\begin{align*}
\left( u_{\varepsilon_k} \right)_t - a_{ij}^{\varepsilon_k} D_{ij}u_{\varepsilon_k} & = \Psi_{\varepsilon_k}(g_{\varepsilon_k}) + f_{\varepsilon_k} \\
& \geq g_{\varepsilon_k}^+ - \varepsilon_k + f_{\varepsilon_k} \\
& = g_{\varepsilon_k} + f_{\varepsilon_k} - \varepsilon_k + g_{\varepsilon_k}^- \\
& = \left( \overline{\psi}_{\varepsilon_k} \right)_t - a_{ij}^{\varepsilon_k} D_{ij}\overline{\psi}_{\varepsilon_k} - \varepsilon_k + g_{\varepsilon_k}^-\\
& \geq \left( \overline{\psi}_{\varepsilon_k} \right)_t - a_{ij}^{\varepsilon_k} D_{ij} \overline{\psi}_{\varepsilon_k} - \varepsilon_k \ \quad \mathrm{in} \ \ V_k.
\end{align*}
Define $\widetilde{u}_{\varepsilon_k}(x,t) := u_{\varepsilon_k}(x,t) + \varepsilon_k t$.
It is a straightforward to check
\begin{equation}
\left\{\begin{array}{rclcc}\
\left( \widetilde{u}_{\varepsilon_k} - \overline{\psi}_{\varepsilon_k} \right)_t - a_{ij}^{\varepsilon_k} D_{ij} \left( \widetilde{u}_{\varepsilon_k} - \overline{\psi}_{\varepsilon_k} \right) & \geq & 0 & \mathrm{in} & V_k,\\
\widetilde{u}_{\varepsilon_k} - \overline{\psi}_{\varepsilon_k} & \geq & 0 & \mathrm{on} & \partial_p V_k,
\end{array}\right.
\end{equation}
where we have used the fact that $u_{\varepsilon_k}=\overline{\psi}_{\varepsilon_k}$ on $\partial_p V_k$.
Then from Lemma \ref{comparison-parabolic}, we have
\begin{equation*}
\widetilde{u}_{\varepsilon_k} - \overline{\psi}_{\varepsilon_k} \geq 0 \ \ \ \mathrm{in} \ \ V_k,
\end{equation*}
and thus
\begin{equation*}
u_{\varepsilon_k} - \overline{\psi}_{\varepsilon_k} \geq -\varepsilon_k t \geq -\varepsilon_k T \ \ \ \mathrm{in} \ \ V_k.
\end{equation*}
Recalling the definition of $V_k$, we see that $u_{\varepsilon_k} - \overline{\psi}_{\varepsilon_k} \geq -\varepsilon_k T$ in $\Omega_T$.
Passing to the limit $k \to \infty$, we discover that $u - \overline{\psi} \geq 0$ a.e. in $\Omega_T$.
Therefore, we conclude that $u - \psi \geq 0$ a.e. in $\Omega_T$.

Finally, we show that $u_t - a_{ij}D_{ij}u = f$ on the set
$\left\lbrace u > \psi \right\rbrace$.
To prove this, we find that for each $m \in \mathbb{N}$, $\Phi_{\varepsilon_k}(u_{\varepsilon_k}-\overline{\psi}_{\varepsilon_k})$ converges to $1$, and $\Psi_{\varepsilon_k}(g_{\varepsilon_k})$ converges to $g^+$ almost everywhere on the set $\left\lbrace u > \overline{\psi} + \frac{1}{m} \right\rbrace$.
Then
\begin{equation*}
u_t - a_{ij}D_{ij}u = - g^+ + g^+ + f = f
\end{equation*}
on the set $\left\lbrace u > \psi \right\rbrace = \left\lbrace u > \overline{\psi} \right\rbrace = \bigcup_{m=1}^{\infty} \left\lbrace u > \overline{\psi} + \frac{1}{m} \right\rbrace$.

As a consequence, we conclude that the problem (\ref{obs prob fully}) has a solution $u \in W^{2,1}L^p_w(\Omega_T)$ and we have the desired estimate
\begin{equation*}
\Norm{u}_{W^{2,1}L^p_w(\Omega_T)} \leq \liminf_{k \to \infty} \Norm{u_{\varepsilon_k}}_{W^{2,1}L^p_w(\Omega_T)} \leq c \left( \Norm{f}_{L^p_w(\Omega_T)} + \Norm{\psi}_{W^{2,1}L^p_w(\Omega_T)} \right).
\end{equation*}
\end{proof}

\begin{remark}
We remark that the uniqueness of a solution to the parabolic obstacle problem (\ref{parabolic obs prob}) is not evident in general.
However, when $D_x a_{ij}$ exist and are bounded, one can obtain the uniqueness of a solution by coerciveness, see for instance \cite{Fri}.
\end{remark}

\bibliographystyle{amsplain}

\begin{thebibliography}{10}


\bibitem{BC1} M. Bramanti and M. Cerutti,
\textit{$W_p^{1,2}$ solvability for the Cauchy-Dirichlet problem for parabolic equations with VMO coefficients},
Comm. Partial Differential Equations \textbf{18} (9-10) (1993), 1735-1763.



\bibitem{BL1} S. Byun and M. Lee,
\textit{Weighted estimates for nondivergence parabolic equations in Orlicz spaces},
J. Funct. Anal. \textbf{269} (8) (2015), 2530-2563.



\bibitem{BL2} S. Byun and M. Lee,
\textit{On weighted $W^{2,p}$ estimates for elliptic equations with BMO coefficients in nondivergence form},
Internat. J. Math. \textbf{26} (1) (2015), 1550001, 28 pp.



\bibitem{BLP} S. Byun, M. Lee and D. Palagachev,
\textit{Hessian estimates in weighted Lebesgue spaces for fully nonlinear elliptic equations},
J. Differential Equations \textbf{260} (5) (2016), 4550-4571.



\bibitem{BP1} S. Byun and D. Palagachev,
\textit{Weighted $L^p$-estimates for elliptic equations with measurable coefficients in nonsmooth domains},
Potential Anal. \textbf{41} (1) (2014), 51-79.



\bibitem{Ca} L.A. Caffarelli,
\textit{Interior a priori estimates for solutions of fully nonlinear equations},
Ann. Math. (2) \textbf{130} (1) (1989), 189-213.



\bibitem{CC} L.A. Caffarelli and X. Cabr\'{e},
\textit{Fully nonlinear elliptic equations},
American Mathematical Society Colloquium Publications, \textbf{43}. Providence, RI: Amer. Math. Soc. 1995.



\bibitem{CCKS} L.A. Caffarelli, M.G. Crandall, M. Kocan and A. \'{S}wi\c{e}ch,
\textit{On viscosity solutions of fully nonlinear equations with measurable ingredients},
Comm. Pure Appl. Math. \textbf{49} (4) (1996), 365-397.



\bibitem{Cam} S. Campanato,
\textit{Sistemi ellittici in forma divergenza. Regolarit\'{a} all'interno},
Quaderni. Scuola Normale Superiore Pisa, Pisa, 1980.



\bibitem{CFL1} F. Chiarenza, M. Frasca and P. Longo,
\textit{Interior $W^{2,p}$ estimates for nondivergence elliptic equations with discontinuous coefficients},
Ricerche Mat. \textbf{40} (1) (1991), 149-168.



\bibitem{CFL2} F. Chiarenza, M. Frasca and P. Longo,
\textit{$W^{2,p}$-solvability of the Dirichlet problem for nondivergence elliptic equations with VMO coefficients},
Trans. Amer. Math. Soc. \textbf{336} (2) (1993), 841-853.



\bibitem{Ch}  S.-K. Chua,
\textit{Some remarks on extension theorems for weighted Sobolev spaces},
Illinois J. Math. \textbf{38} (1) (1994), 95-126.



\bibitem{CR}  R.R. Coifman and R. Rochberg,
\textit{Another characterization of BMO},
Proc. Amer. Math. Soc. \textbf{79} (2) (1980), 249-254.



\bibitem{Es} L. Escauriaza,
\textit{$W^{2,n}$ a priori estimates for solutions to fully nonlinear equations},
Indiana Univ. Math. J. \textbf{42} (2) (1993), 413-423.



\bibitem{FS} A. Figalli and H. Shahgholian,
\textit{A general class of free boundary problems for fully nonlinear elliptic equations},
Arch. Ration. Mech. Anal. \textbf{213} (1) (2014), 269-286.



\bibitem{Fri} A. Friedman,
\textit{Variational principles and free-boundary problems},
A Wiley-Interscience Publication. Pure and Applied Mathematics. John Wiley \& Sons, Inc., New York, 1982.



\bibitem{HHH} R. Haller-Dintelmann, H. Heck and M. Hieber,
\textit{$L^p$-$L^q$ estimates for parabolic systems in non-divergence form with VMO coefficients},
J. London Math. Soc. (2) \textbf{74} (3) (2006), 717-736.



\bibitem{IM} E. Indrei and A. Minne,
\textit{Regularity of solutions to fully nonlinear elliptic and parabolic free boundary problems},
Ann. Inst. H. Poincar\'{e} Anal. Non Linéaire \textbf{33} (5) (2016), 1259-1277.



\bibitem{Ki} T. Kilpel\"{a}inen,
\textit{Smooth approximation in weighted Sobolev spaces},
Comment. Math. Univ. Carolin. \textbf{38} (1) (1997), 29-35.



\bibitem{Li} G.M. Lieberman,
\textit{Second order parabolic differential equations},
World Sci. Publ. Co., Inc., River Edge, NJ, 1996.



\bibitem{PSU} A. Petrosyan, H. Shahgholian and N. Uraltseva,
\textit{Regularity of free boundaries in obstacle-type problems},
Graduate Studies in Mathematics, 136. American Mathematical Society, Providence, RI, 2012.



\bibitem{St} E.M. Stein,
\textit{Harmonic analysis: real-variable methods, orthogonality, and oscillatory integrals},
Princeton Mathematical Series, 43. Monographs in Harmonic Analysis, III, Princeton University Press, Princeton, NJ, 1993.



\bibitem{Te} K.H. Teka,
\textit{The obstacle problem for second order elliptic operators in nondivergence form},
Thesis (Ph.D.)-Kansas State University. 2012.



\bibitem{Tu} B.O. Turesson,
\textit{Nonlinear potential theory and weighted Sobolev spaces},
Lecture Notes in Mathematics, 1736. Springer-Verlag, Berlin, 2000.



\bibitem{Ur} J.M. Urbano,
\textit{The method of intrinsic scaling. A systematic approach to regularity for degenerate and singular PDEs},
Lecture Notes in Mathematics, 1930. Springer-Verlag, Berlin, 2008.



\bibitem{WY1} L. Wang and F. Yao,
\textit{Higher-order nondivergence elliptic and parabolic equations in Sobolev spaces and Orlicz spaces},
J. Funct. Anal. \textbf{262} (8) (2012), 3495-3517.



\bibitem{Wi} N. Winter,
\textit{$W^{2,p}$ and $W^{1,p}$-estimates at the boundary for solutions of fully nonlinear, uniformly elliptic equations},
Z. Anal. Anwend. \textbf{28} (2) (2009), 129-164.



\end{thebibliography}

\end{document}